\documentclass[12pt]{article}

\usepackage{amssymb}%
\usepackage{amsmath}%
\usepackage{amsthm}%
\usepackage{graphicx}%
\usepackage{hyperref}%
\usepackage{xcolor}%
\usepackage{mathrsfs}%

\allowdisplaybreaks%

{\theoremstyle{plain}%
 \newtheorem{theorem}{Theorem}

}
{\theoremstyle{remark}

}
{\theoremstyle{definition}
\newtheorem{definition}{Definition}
\newtheorem{example}{Example}
}

\begin{document}

\begin{center}
{\large Abelian-normal decimal expansions}

 \ 

{\textsc{John M. Campbell}} 

 \ 

{\small Department of Mathematics and Statistics}\\
{\small Dalhousie University}\\ 
{\small Halifax, NS B3H 4R2}\\ 
{\small Canada} \\
{\small \href{mailto: jh241966@dal.ca}{\tt jh241966@dal.ca}}\\

\end{center}

\begin{abstract}
 A landmark paper on normal numbers due to Champernowne [\emph{J.\ London Math.\ Soc.}, 1933] provided the first explicit 
 construction of a normal number. Since many research works have concerned normality-preserving selection rules and operations on 
 the sequence of digits of a given normal number, this leads us to introduce the concept of an \emph{abelian-normal number}, in such 
 a way so that subwords that are equivalent up to permutations are counted as being equivalent, and so that the number of occurrences 
 of equivalent subwords appearing, up to a given point in a decimal expansion, is then normalized by the size of the associated 
 equivalence class. This may be thought of as giving rise to an interdisciplinary area linking combinatorics on words and probabilistic 
 number theory. Since every normal number is abelian-normal, this raises the question as to whether or not there exists an 
 abelian-normal number that is not normal. We succeed in constructing such a constant in an explicit way. 
\end{abstract}

\vspace{0.1in}


\vspace{0.1in}


\section{Introduction}
 For a fixed base $\mathcal{B} \geq 2$, for a positive value $\alpha < 1$, and for a nonempty block $E$ of base-$\mathcal{B}$ digits, we write 
 $A_{E}(\alpha, n) = A_{E, \mathcal{B}}(\alpha, n)$ in place of the number of occurrences of $E$ in the first $n$ base-$\mathcal{B}$ digits
 of $\alpha$ (past the decimal point), adopting notation from Sz\"usz and Volkmann \cite{SzuszVolkmann1994}, and adopting 
 the convention whereby words are counted with 
 overlaps allowed (and adopting the standard convention whereby a base-$\mathcal{B}$ expansion is such that it is 
 not eventually equal to an infinite sequence of digits of the form $\mathcal{B} - 1$). 
 Let $\ell(E)$ denote the \emph{length} of $E$, i.e., the number of digits occurring in $E$, with multiplicities counted. 
 The value $\alpha$ is said to be \emph{normal} in base $\mathcal{B}$ if 
\begin{equation}\label{definitionnormal}
 \lim_{n \to \infty} \frac{A_{E}(\alpha, n)}{n} = \frac{1}{\mathcal{B}^{\ell(E)}} 
\end{equation}
 for all nonempty blocks $E$, and similarly for real numbers $\alpha \geq 1$. 
 While the concept of a normal number was introduced in 1909 by Borel~\cite{Borel1909}, the given definition of a 
 normal number is reminiscent of the relatively recent discipline known as combinatorics on words, which often concerns the enumeration 
 of blocks occurring in a given sequence or word. This leads us to consider new and interdisciplinary subjects given 
 by how recent developments in combinatorics on words could be used to extend the definition of a normal number. 

 Let $\mathbb{N}$ denote the set of positive integers and let $\mathbb{N}_{0} = \mathbb{N} \cup \{ 0 \}$. For a given sequence 
 $\text{{\bf x}}$ with a finite codomain, i.e., a function on $\mathbb{N}$ with a finite codomain, we let this be identified with an infinite 
 word. The \emph{factor complexity function} $\rho_{\text{{\bf x}}}\colon \mathbb{N}_{0} \to \mathbb{N}$ on $\text{{\bf x}}$ maps $n \in 
 \mathbb{N}_{0}$ to the number of length-$n$ factors appearing in $\text{{\bf x}}$. Similarly, the \emph{abelian complexity function} 
 $\rho_{\text{{\bf x}}}^{\operatorname{ab}}\colon \mathbb{N}_{0} \to \mathbb{N}$ maps $n \in \mathbb{N}_{0}$ to the number, up to 
 equivalence by permutations of words of equal length, of length-$n$ factors occurring in $\text{{\bf x}}$. The significance of abelian 
 complexity functions in the field of combinatorics on words leads us to consider how similarly defined functions could be applied to 
 extend the definition of a normal number. This leads us to introduce the concept of an \emph{abelian-normal number}. 

 Let the equivalence relation $\sim$ denote equivalence 
 of finite words up to permutations. For a finite word $w$, 
 we then let $[w]_{\sim}$ denote the equivalence
 class consisting of all permutations of $w$. 

\begin{definition}\label{definitionBcount}
 Define
 $$ B_{E}(\alpha, n) := \sum_{F \in [E]_{\sim}} A_{F}(\alpha, n). $$
 \end{definition}

\begin{example}
 The first explicit construction of a normal number is due to Champernowne in 1933 \cite{Champernowne1933} and is given by the
 real number
\begin{equation}\label{displayC10}
 C_{10} = 0.12345678910111213141516171819202122232425262728293\ldots 
\end{equation}
 given by concatenating, after the decimal point, the base-10 expansions of consecutive positive integers. 
 Among the 50 digits past the decimal point displayed in \eqref{displayC10}, 
 we find that 
 $A_{12}(C_{10}, 50) = 3$, whereas 
 $B_{12}(C_{10}, 50) = 5$, as there are $2$ occurrences of the permutation $21$ of the block $12$. 
\end{example}

\begin{definition}\label{abnormal}
 For a base $\mathcal{B} \geq 2$, 
 a positive value $\alpha$ is said to be \emph{abelian-normal} in base $\mathcal{B}$ if 
\begin{equation}\label{displayabnormal}
 \lim_{n \to \infty} \frac{1}{ \left| [E]_{\sim} \right| } \frac{B_{E}(\alpha, n)}{n} = \frac{1}{\mathcal{B}^{\ell(E)}} 
\end{equation}
 for all  nonempty   blocks $E$. 
\end{definition}

 It is immediate that every normal number is abelian-normal. Indeed, we see that 
\begin{align*}
 \lim_{n \to \infty} \frac{B_{E}(\alpha, n)}{n} 
 & = \sum_{F \in [E]_{\sim}} \lim_{n \to \infty} \frac{A_{F}(\alpha, n)}{n} \\ 
 & = \frac{ \left| [E]_{\sim} \right| }{\mathcal{B}^{\ell(E)}}. 
\end{align*}
 This raises the question 
 as to whether or not there is a number that is abelian-normal but not normal. 
 We succeed, based on our extensive interactions with GPT-5.6 Pro, 
 in explicitly constructing a 
 number that is abelian-normal but not normal. 

\section{An abelian-normal, non-normal number}\label{sectionabnormal}
 Let
\begin{equation}\label{C100}
 C_{100}=0.c_1c_2\ldots
\end{equation}
 denote the base-$100$ Champernowne constant, with
$c_i\in\{0,1,\ldots,99\}$ for all positive integers $i$; that is,
$C_{100}$ is formed by concatenating the base-$100$ digit expansions of
consecutive positive integers. 
 It is known that the base-$b$ Champernowne constant $C_b$, which is 
 formed by concatenating the base-$b$ expansions of consecutive positive integers after the decimal point, 
 is normal in base $b \geq 2$, 
 and, one way of proving this would be
 as an immediate consequence of a construction due to 
 Sz\"usz and Volkmann \cite{SzuszVolkmann1994}. 
 We thus have that $C_{100}$ is normal in base $100$. 
 
For each base-$100$ digit $u$, write $u=10a+b$, where    $a,b\in\{0,1,\ldots,9\}$. We then define the length-$2$ decimal   word  
 $\phi(u)$     by
\begin{equation}\label{definephi}
 \phi(10a+b)=
 \begin{cases}
 01, & (a,b)=(1,0),\\
 12, & (a,b)=(2,1),\\
 20, & (a,b)=(0,2),\\
 ab, & \text{otherwise.}
 \end{cases}
\end{equation}
 We identify a length-2 word with the ordered pair consisting of
 the first and second characters, in order, and vice-versa. 
 We also write
$\phi(u)=\phi_1(u)\phi_2(u)$, where $\phi_1(u)$ and $\phi_2(u)$ are the
first and second decimal digits of $\phi(u)$, respectively.

Being consistent with the notation in \eqref{C100}, we define
\begin{equation}\label{defineXi10}
 \Xi_{10}=0.\phi(c_1)\phi(c_2)\phi(c_3)\cdots.
\end{equation}
The symbols $0,1,\ldots,99$ occurring in $c_1c_2\cdots$ are treated as
distinct base-$100$ labels before $\phi$ is applied. Thus,
$$
 \Xi_{10}
 =0.01{\big|}20{\big|}03{\big|}04{\big|}05{\big|}06{\big|}
 07{\big|}08{\big|}09{\big|}01{\big|}11{\big|}12{\big|}\cdots.
$$

\begin{theorem}\label{theoremXi10}
 The constant $\Xi_{10}$ is abelian-normal in base $10$ but not normal
 in base $10$.
\end{theorem}

\begin{proof}
   Let $U$ be uniformly distributed over the labels in   $\{0,1,\ldots,99\}$, and write $XY=\phi(U)$,     with the understanding that $X$ and  
  $Y$, respectively, 
 are equal to the first and second digits (written as single-character words) of $\phi(U)$, 
 and where $XY$ denotes the concatenation of these words. 
 It follows in an immediate way from
\eqref{definephi} that
\begin{equation}\label{PrXYab}
 \operatorname{Pr}(XY=ab)=
 \begin{cases}
 \dfrac{2}{100}, & ab\in\{01,12,20\},\\
 0, & ab\in\{10,21,02\},\\
 \dfrac{1}{100}, & \text{otherwise.}
 \end{cases}
\end{equation}
 While the distribution in \eqref{PrXYab} 
 is not uniform, each of the associated 
 coordinate distributions is uniform, with 
\begin{equation}\label{twouniform}
 \operatorname{Pr}(X=d)=\operatorname{Pr}(Y=d)=\frac{1}{10} 
\end{equation}
 for $ d \in \{ 0, 1, \ldots, 9 \}$. 
 Indeed, in the first coordinate, the three exceptional substitutions 
 (indicated in the definition of $\phi$ in \eqref{definephi}) 
 replace $1$ by $0$, $2$ by $1$, and $0$ by $2$, whereas, in the second
coordinate, these substitutions replace $0$ by $1$, $1$ by $2$, and $2$ by $0$.
Consequently, the number of occurrences of each digit in either
coordinate is unchanged.

For commuting indeterminates $z_0,z_1,\ldots,z_9$, write
$\mathbf z=(z_0,z_1,\ldots,z_9)$ and set
\begin{equation*}
 S(\mathbf z)=\frac{z_0+z_1+\cdots+z_9}{10}.
\end{equation*}
We also set
\begin{equation*}
 P_j(\mathbf z)
 =\frac{1}{100}\sum_{u=0}^{99}z_{\phi_j(u)} 
\end{equation*}
 for $j \in \{ 1, 2 \}$
 and 
 $$ Q(\mathbf z)
 =\frac{1}{100}\sum_{u=0}^{99}
 z_{\phi_1(u)}z_{\phi_2(u)}.
 $$
 From \eqref{twouniform}, we find that 
\begin{equation}\label{singlecoordinateidentities}
 P_1(\mathbf z)=P_2(\mathbf z)=S(\mathbf z).
\end{equation}
Moreover, each of the three exceptional substitutions in
\eqref{definephi} reverses a pair of digits. Since the given 
indeterminates commute, the corresponding monomial is unchanged.
 As a consequence, we have that 
\begin{equation*}
 Q(\mathbf z)
 =\frac{1}{100}\sum_{a,b=0}^{9}z_az_b
 = S^2(\mathbf z).
\end{equation*}

 We next record a consequence of the base-$100$ normality of
$c_1c_2\cdots$ that is required for our purposes. For every fixed positive
integer $s$ and every function
$F\colon\{0,1,\ldots,99\}^{s}\longrightarrow\mathbb C$, we claim that 
\begin{equation}\label{normalaveraging}
 \lim_{N\to\infty}\frac{1}{N}
 \sum_{i=1}^{N}F(c_i,c_{i+1},\ldots,c_{i+s-1}) =\frac{1}{100^{s}}
 \sum_{u_1,\ldots,u_s=0}^{99}F(u_1,\ldots,u_s). 
\end{equation} 
 This may be seen by grouping the terms on the left-hand side 
 of \eqref{normalaveraging}
 according to each possible
length-$s$ block, as every such block is of limiting frequency $100^{-s}$. 
 
 Now, write $$  x_1 x_2 x_3 \cdots=\phi(c_1)\phi(c_2)\phi(c_3)\cdots, $$ so that
\begin{equation}\label{xcoordinates}
 x_{2i-1}=\phi_1(c_i)
 \quad\text{and}\quad
 x_{2i}=\phi_2(c_i) 
\end{equation} 
 for each positive integer $i$. 
For a finite decimal word $w=w_1w_2\cdots w_m$, let
\begin{equation}\label{defineM}
 M(w;\mathbf z)
 =\prod_{j=1}^{m}z_{w_j}
 =\prod_{d=0}^{9}z_d^{\nu_d(w)},
\end{equation}
where $\nu_d(w)$ is the number of occurrences of the digit $d$ in $w$.
All polynomial limits below are understood coefficientwise.

 If $m=2r$ (for $r \geq 1$), then \eqref{normalaveraging} gives 
\begin{equation}\label{oddstartevenlength}
 \lim_{N\to\infty}\frac{1}{N}
 \sum_{i=1}^{N}
 M(x_{2i-1}\cdots x_{2i+2r-2};\mathbf z) 
 = Q^r(\mathbf z) 
 = S^{2r}(\mathbf z).
\end{equation}
 If $m=2r+1$ (for $r \geq 0$), then \eqref{normalaveraging} gives 
\begin{equation}%keepitlabeled
 \lim_{N\to\infty}\frac{1}{N}
 \sum_{i=1}^{N}
 M(x_{2i-1}\cdots x_{2i+2r-1};\mathbf z) 
 =Q^r(\mathbf z) P_1(\mathbf z)
 = S^{2r+1}(\mathbf z).
\end{equation}
 If $m=2r$ (for $r \geq 1$), then \eqref{normalaveraging} gives 
\begin{equation}%keepitlabeled 
 \lim_{N\to\infty}\frac{1}{N}
 \sum_{i=1}^{N}
 M(x_{2i}\cdots x_{2i+2r-1};\mathbf z) 
 = P_2(\mathbf z) Q^{r-1}(\mathbf z) P_1(\mathbf z)
 = S^{2r}(\mathbf z). 
\end{equation} 
 If $m=2r+1$ (for $r \geq 0$), then \eqref{normalaveraging} gives 
\begin{equation}\label{evenstartoddlength}
 \lim_{N\to\infty}\frac{1}{N}
 \sum_{i=1}^{N}
 M(x_{2i}\cdots x_{2i+2r};\mathbf z) 
 = P_2(\mathbf z) Q^r(\mathbf z) 
 = S^{2r+1}(\mathbf z).
\end{equation} 
 So, regardless of whether a length-$m$ block begins at an odd position or at an even position, the limit of each 
 displayed average among \eqref{oddstartevenlength}--\eqref{evenstartoddlength} 
 is $S^m(\mathbf z)$.  Moreover, among 
 the $n-m+1$ possible starting positions of length-$m$ blocks, the numbers of odd and 
 even starting positions are 
\begin{equation}\label{dividebyn} 
 \left\lceil\frac{n-m+1}{2}\right\rceil
 \quad\text{and}\quad
 \left\lfloor\frac{n-m+1}{2}\right\rfloor, 
\end{equation}
 respectively, and, furthermore, since $m$ is fixed, each of these quantities in \eqref{dividebyn}, after division by $n$,
tends to $\frac{1}{2}$. It follows from this and from 
 \eqref{oddstartevenlength}--\eqref{evenstartoddlength} that 
\begin{equation}\label{allstartpositions}
 \lim_{n\to\infty}\frac{1}{n}
 \sum_{j=1}^{n-m+1}M(x_jx_{j+1}\cdots x_{j+m-1};\mathbf z)
 = S^m(\mathbf z) 
\end{equation}
 for every fixed $m\geq1$. 

 Now, let $E$ be a (nonempty) decimal block of length $m$.   By the definition of $B_E$, the coefficient of     $z_0^{\nu_0(E)}  
  z_1^{\nu_1(E)}\cdots z_9^{\nu_9(E)}$ in the polynomial   $$  \frac{1}{n}
 \sum_{j=1}^{n-m+1}M(x_jx_{j+1}\cdots x_{j+m-1};\mathbf z)
$$ is $\frac{B_E(\Xi_{10},n)}{n}$. On the other hand,
\begin{equation}\label{takecoefficient}
 \left[z_0^{\nu_0(E)}z_1^{\nu_1(E)}\cdots z_9^{\nu_9(E)}\right]
 S(\mathbf z)^m
 =\frac{m!}{\nu_0(E)!\nu_1(E)!\cdots\nu_9(E)!}\frac{1}{10^m}.
\end{equation}
Since
\begin{equation}\label{equivalenceclasssize}
 \left|[E]_{\sim}\right|
 =\frac{m!}{\nu_0(E)!\nu_1(E)!\cdots\nu_9(E)!},
\end{equation}
taking the  coefficient indicated   in \eqref{allstartpositions} gives us that  
$$
 \lim_{n\to\infty}
 \frac{1}{\left|[E]_{\sim}\right|}
 \frac{B_E(\Xi_{10},n)}{n}
 =\frac{1}{10^m}
 =\frac{1}{10^{\ell(E)}}.
$$
   We thus have that   $\Xi_{10}$ is abelian-normal in base $10$.  

 So, it remains to show that $\Xi_{10}$ is not normal. In this direction, we begin by considering the ordered decimal block $01$. At starting 
 positions of the form $2i-1$, the block of length $2$ is exactly $\phi(c_i)$. The word $01$ has two
preimages under $\phi$, namely the base-$100$ symbols $1$ and $10$.
Since every base-$100$ symbol has limiting frequency $1/100$, we have
\begin{equation}\label{aligned01frequency}
 \lim_{N\to\infty}\frac{1}{N}
 \#\{1\leq i\leq N:x_{2i-1}x_{2i}=01\}
 =\frac{2}{100}.
\end{equation}
 At starting positions of the form $2i$, the relevant block is $\phi_2(c_i)\phi_1(c_{i+1})$. By \eqref{normalaveraging} with $s=2$, and by 
 \eqref{singlecoordinateidentities}, its limiting frequency is
\begin{align}
\begin{split}
 &\frac{1}{100^2}
 \#\{(u,v)\in\{0,1,\ldots,99\}^2:
 \phi_2(u)=0,\ \phi_1(v)=1\} \\
 &\qquad=
 \frac{1}{100^2}
 \#\{u:\phi_2(u)=0\}
 \#\{v:\phi_1(v)=1\}
 =\frac{10\cdot10}{100^2}
 =\frac{1}{100}.
\end{split}\label{crossing01frequency}
\end{align}
  Finally, starting positions of the forms $2i-1$ and $2i$ each have  asymptotic density $\frac{1}{2}$.  Combining
 \eqref{aligned01frequency} and \eqref{crossing01frequency}, we obtain
$$
 \lim_{n\to\infty}\frac{A_{01}(\Xi_{10},n)}{n}
 =\frac{1}{2}\frac{2}{100}
 +\frac{1}{2}\frac{1}{100}
 =\frac{3}{200}, 
$$
 in contrast to the value $\frac{1}{100}$   required for the frequency for the
block $01$ in a normal base-$10$ expansion. We thus have that $\Xi_{10}$ is not
normal in base $10$.
\end{proof}

\section{Conclusion} 
 A major open problem concerning normal numbers is given by whether or not ``naturally occurring'' constants (as opposed to 
 ``artificially constructed'' constants given by concatenating digit expansions) such as $\pi$ and $e$ are normal. This leads us to 
 conclude by encouraging the pursuit of research efforts based on the determination as to whether or not constants such as $\pi$ and 
 $e$ are abelian-normal. 

 \subsection*{Acknowledgements}
 The author thanks Jean-Paul Allouche and Ver\'{o}nica Becher for useful feedback concerning this research paper. 
 The author acknowledges extensive interactions with GPT-5.6 Pro during the exploratory and proof-development stages of this work. All 
 AI-generated suggestions were substantially revised, corrected, and independently verified by the author, who assumes full responsibility 
 for the mathematical content.

\end{document}